\documentclass[12pt]{amsart}

\usepackage{amsxtra,amssymb,amsthm,amsmath,amscd,url,listings}
\usepackage[utf8]{inputenc}
\usepackage{eucal}
\usepackage{fullpage}
\usepackage{scrtime}
\usepackage{hyperref}

\usepackage{graphicx}

\newcounter{point}

\renewcommand{\leq}{\leqslant}
\renewcommand{\geq}{\geqslant}

\numberwithin{equation}{section}

\newcommand{\uple}[1]{\text{\boldmath${#1}$}}

\newcommand{\Cc}{\mathbf{C}}

\newcommand{\Aa}{\mathbf{A}}

\newcommand{\Zz}{\mathbf{Z}}

\newcommand{\Rr}{\mathbf{R}}

\newcommand{\Fp}{{\mathbf{F}_p}}

\newcommand{\Fpt}{{\mathbf{F}^\times_p}}

\newcommand{\Ff}{\mathbf{F}}

\newcommand{\proba}{\mathbf{P}}
\newcommand{\expect}{\mathbf{E}}

\newcommand{\mods}[1]{\,(\mathrm{mod}\,{#1})}

\DeclareMathOperator{\hypk}{Kl}
\DeclareMathOperator{\birch}{Bi}

\newcommand{\HYPK}{\mathcal{K}\ell}
\newcommand{\kpath}{\mathrm{K}}
\newcommand{\tkp}{\tilde{K}}
\newcommand{\ekp}{\mathcal{K}}
\newcommand{\st}{\mathrm{ST}}

\newcommand{\ra}{\rightarrow}
\newcommand{\lra}{\longrightarrow}

\DeclareMathOperator{\Imag}{Im}
\DeclareMathOperator{\Reel}{Re}

\newcommand{\eps}{\varepsilon}
\renewcommand{\rho}{\varrho}

\DeclareMathOperator{\SL}{SL}

\DeclareMathOperator{\SU}{SU}

\newcommand{\demi}{{\textstyle{\frac{1}{2}}}}

\newcommand{\sheaf}[1]{\mathcal{{#1}}}

\DeclareMathSymbol{\gena}{\mathord}{letters}{"3C}
\DeclareMathSymbol{\genb}{\mathord}{letters}{"3E}

\def\multsum{\mathop{\sum\cdots \sum}\limits}

\theoremstyle{plain}
\newtheorem{theorem}{Theorem}[section]
\newtheorem{lemma}[theorem]{Lemma}

\newtheorem{problem}[theorem]{Problem}
\newtheorem{proposition}[theorem]{Proposition}
\newtheorem*{proposition*}{Proposition}

\theoremstyle{remark}

\theoremstyle{definition}

\newtheorem{remark}[theorem]{Remark}

\renewcommand{\geq}{\geqslant}
\renewcommand{\leq}{\leqslant}

\newcommand\sumsum{\mathop{\sum\sum}\limits}

\begin{document}

\title{Kloosterman paths and the shape of exponential sums}
 
 \author{Emmanuel Kowalski}
\address{ETH Z\"urich -- D-MATH\\
  R\"amistrasse 101\\
  CH-8092 Z\"urich\\
  Switzerland} \email{kowalski@math.ethz.ch}


\author{William F. Sawin}
\address{Princeton University, Fine Hall, Washington Road, NJ, USA}
\email{wsawin@math.princeton.edu}

\date{\today,\ \thistime} 

\begin{abstract} We consider the distribution of the polygonal paths
  joining partial sums of classical Kloosterman sums $\hypk_p(a)$, as
  $a$ varies over $\Fpt$, and as $p$ tends to infinity. Using
  independence of Kloosterman sheaves, we prove convergence in the
  sense of finite distributions to a specific random Fourier
  series. We also consider Birch sums, for which we can establish
  convergence in law in the space of continuous functions. We then
  derive some applications.
\end{abstract}




\thanks{E.K. was supported partly by a DFG-SNF lead agency program
  grant (grant 200021L\_153647)}
  
  \thanks{This material is based upon work supported by the National Science Foundation Graduate Research Fellowship under Grant No. DGE-1148900}

\maketitle

\begin{flushright}
  Dedicated to the memory of Marc Yor\\
  \smallskip
  \textit{L'avenir est au hasard}\\
  (Jacques Brel)
\end{flushright}

\section{Introduction}

For a prime number $p$ and $a\in\Fp$, we denote by
$$
\hypk_p(a)=\frac{1}{\sqrt{p}}\sum_{x=1}^{p-1}\psi_p(ax+\bar{x})
$$
the normalized classical Kloosterman sum, where
$\psi_p(z)=e(z/p)=e^{2i\pi z/p}$ is the standard additive character
modulo $p$ and $\bar{x}$ denotes the inverse of $x$ modulo $p$.
\par
Motivated partly by curiosity, arising to a large extent from staring
at the corresponding plots (see~\cite{kloostermania}), we consider in
this paper the geometric properties of the \emph{Kloosterman paths} in
the complex plane. These are defined as follows: for each prime $p$
and integer $a\in\Fpt$, we first let $\gamma_p(a)$ denote the
polygonal path obtained by concatenating the closed segments
$[z_j,z_{j+1}]$ joining the successive partial sums
$$
z_j=\frac{1}{\sqrt{p}}\sum_{1\leq x\leq
  j}\psi_p(ax+\bar{x}),\quad\quad
z_{j+1}=\frac{1}{\sqrt{p}}\sum_{1\leq x\leq j+1}\psi_p(ax+\bar{x}),
$$
for $0\leq j\leq p-2$. We then define a continuous map
$$
t\mapsto K_p(t,a)
$$ 
for $t\in [0,1]$ by parameterizing the path $\gamma_p(a)$, each
segment $[z_j,z_{j+1}]$ being parameterized linearly by $t\in
[j/(p-1),(j+1)/(p-1)]$.
Figure~\ref{fig-1} shows the plot of $t\mapsto K_{10007}(t,1)$, which
should explain clearly the meaning of the definition.
\par
\begin{figure}\label{fig-1}
\caption{Plot of $t\mapsto K_{10007}(t,1)$}
\includegraphics[height=7cm]{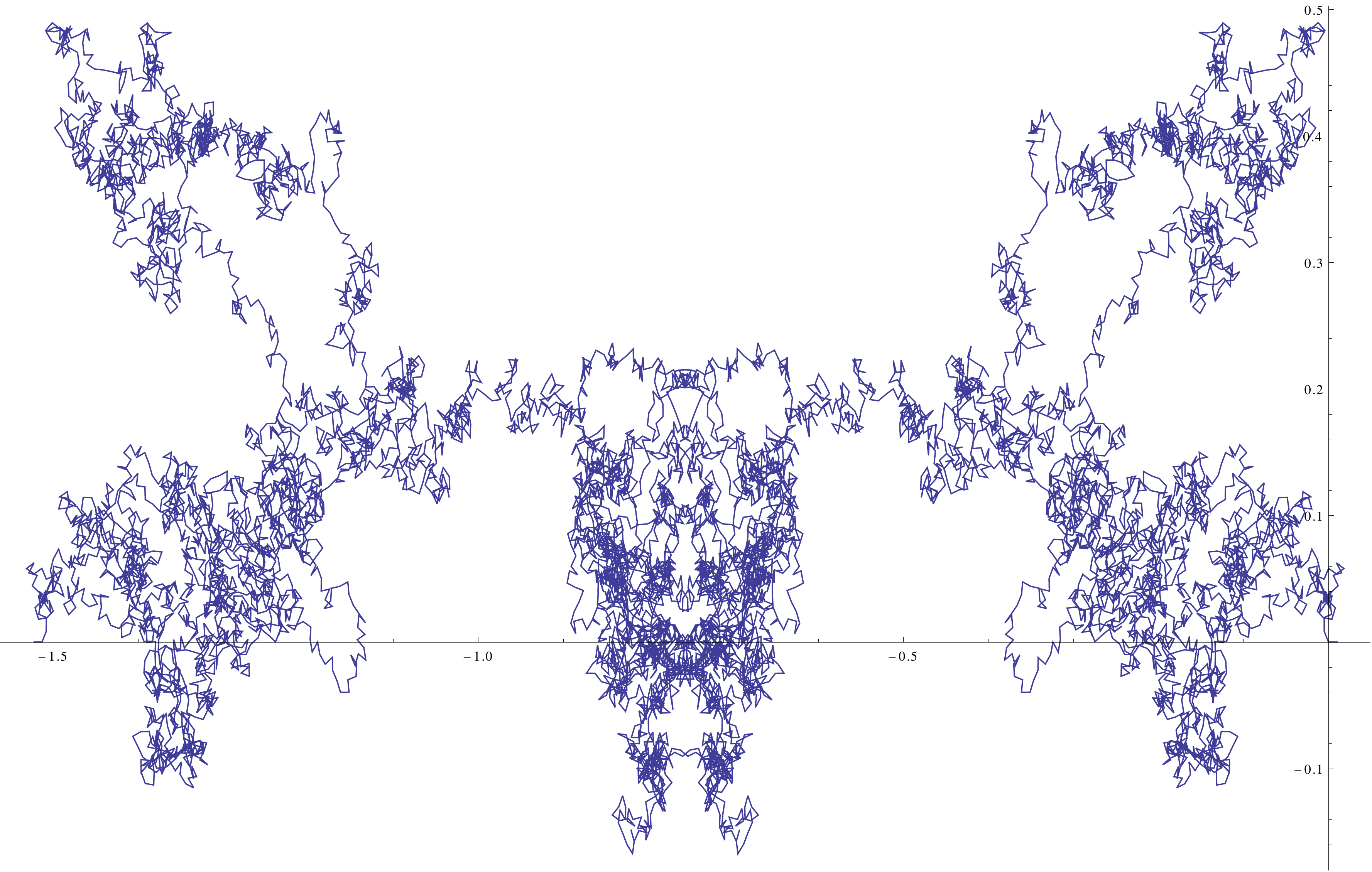}
\end{figure}
\par
We view each $a\mapsto K_p(t,a)$ as a random variable on the finite
probability space
$$
(\Fpt,\text{ uniform probability measure}),
$$
and frequently write simply $K_p(t)$ for this random variable. Thus
$(K_p(t))_{t\in [0,1]}$ is a (simple) stochastic process.
\par
We will use the computation of monodromy groups of Kloosterman sheaves
to find the limiting distribution of $(K_p(t))$ as $p\ra +\infty$, in
the sense of convergence of finite distributions. 

\begin{theorem}\label{th-main}
  Let $(\st_h)_{n\in\Zz}$ denote independent identically distributed
  random variables with distribution equal for all $h\in\Zz$ to the
  Sato-Tate measure
$$
\mu_{ST}=\frac{1}{\pi}\sqrt{1-(x/2)^2}dx
$$
on $[-2,2]$.
\par
\emph{(1)} The random series
$$
\kpath(t)=\sum_{h\in\Zz} \frac{e^{2\pi i ht}-1}{2\pi ih}\st_h
$$
converges almost surely and in law, taking symmetric partial sums,
where the term $h=0$ is interpreted as $t\st_0$. Its limit, as a
random function, is almost surely continuous. In addition, it is
almost surely nowhere differentiable.
\par
\emph{(2)} The sequence of processes $(K_p(t))_{t\in [0,1]}$ converges
to the process $(\kpath(t))_{t\in [0,1]}$, in the sense of convergence
of finite distributions, i.e., for every $k\geq 1$, for every
$k$-tuple $0\leq t_1<\cdots<t_k\leq 1$, the vectors
$$
(K_p(t_1),\ldots,K_p(t_k))
$$
converge in law, as $p\ra +\infty$, to
$$
(\kpath(t_1),\ldots, \kpath(t_k)).
$$
\end{theorem}

\begin{remark}
  (1) Lehmer~\cite{lehmer} and Loxton~\cite{loxton1,loxton2}
  considered the ``graphs'' of various exponential sums, which are the
  analogues of the paths $t\mapsto K_p(t,a)$, but not necessarily over
  finite fields (see for instance the pictures
  in~\cite[p. 127]{lehmer} and~\cite[p. 154--155]{loxton1}).  In
  particular, in~\cite[p 16]{loxton2}, Loxton mentions briefly that
  the paths of Kloosterman sums ``seems to be absolutely chaotic''.
  Our result indicates one precise way in which this is true (or
  false).
\par
(2) Figure~\ref{fig-2} shows a sample simulation of the process
$(\kpath(t))_{t\in [0,1]}$ with $N=10000$ steps, obtained as follows:
values at $j/N$ are simulated for $0\leq j\leq 9999$, by summing the
partial sum of the random series between $-5000$ and $5000$ (using
samples of a Sato-Tate distribution), and then the corresponding
points are interpolated linearly as in the Kloosterman paths.

\begin{figure}
\label{fig-2}
\caption{A sample of the random Fourier series $\kpath(t)$}
\includegraphics[height=7cm]{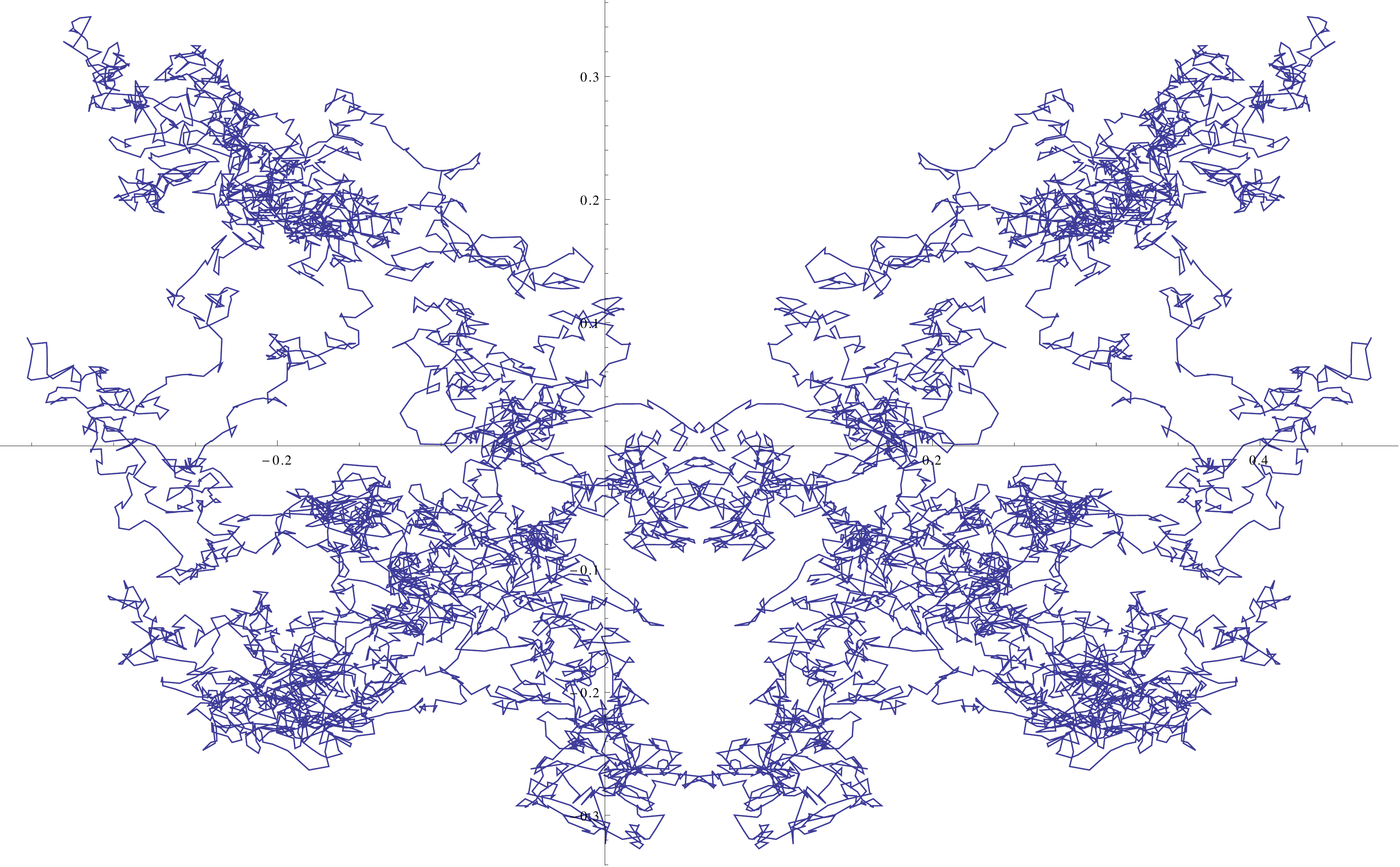}
\end{figure}

%
\end{remark}

Intuitively, the statement of Theorem~\ref{th-main} is not quite
satisfactory, because we may wish to have convergence in law of the
processes as random elements in the space $C([0,1])$ of continuous
functions from $[0,1]$ to $\Cc$. We will see that some highly natural
conjectures concerning short exponential sums lead to this conclusion
(see Section~\ref{sec-tight}). Moreover, we can show unconditionally
such a stronger convergence in law in two cases:
\par
(1) If we consider the family of Kloosterman paths also on average
over all additive characters $x\mapsto \psi_p(\alpha x)$ for
$\alpha\in\Fpt$;
\par
(2) For the partial sums of the family of cubic exponential sums
(sometimes known as Birch sums).
\par
We will also see that this stronger result extends (with possibly
different limiting distribution) to a number of other cases.
\par
We introduce some notation for this purpose. For each prime $p$, we
assume given a probability space $\Omega_p$, and a family
$\mathcal{X}_p=(\xi_p(x))_{x\in\Fp}$
of complex-valued random variables on $\Omega_p$. We then define the
associated path process $K_p^{\mathcal{X}}(t)$ for $t\in [0,1]$
by parameterizing by $t\in [0,1]$ the
polygonal path joining the successive partial sums
$$
\omega\mapsto \frac{1}{\sqrt{p}}\sum_{0\leq x\leq j-1} \xi_p(x,\omega)
$$
for $\omega\in\Omega_p$ and $0\leq j\leq p-1$, which we interpolate
between $j/p$ and $(j+1)/p$. 
\par
We can also have a family $\mathcal{X}_p=(\xi_p(x))_{x\in\Fpt}$
parameterized by $\Fpt$ (as in the case of Theorem~\ref{th-main}), and
then we define $K_p^{\mathcal{X}}(t)$ for $t\in [0,1]$ as in that
case, interpolating between $j/(p-1)$ and $(j+1)/(p-1)$.
\par
We view these processes as $C([0,1])$-valued random processes. We will
study them in a special case, and mention possible generalizations at
the end of the paper.

\begin{theorem}\label{th-main2}
  For $\Omega_p=\Fpt$ with the uniform probability measure and
$$
\xi_p(x,a)=\psi_p(x^3+ax),
$$
for $a\in\Fpt$ and $x\in\Fp$, the processes $(K^{\mathcal{X}}_p(t))$
converge to the process $(\kpath(t))_{t\in [0,1]}$ in the sense of
convergence in law in $C([0,1])$.
\end{theorem}

\begin{remark}
  In Figure~\ref{fig-3}, we plot the function $t\mapsto
  K^{\mathcal{X}}_{10007}(t,1)$ for this choice of $\mathcal{X}$.
\begin{figure}\label{fig-3}
  \caption{Path of Birch sum $\birch_{10007}(1)$}
\includegraphics[height=7cm]{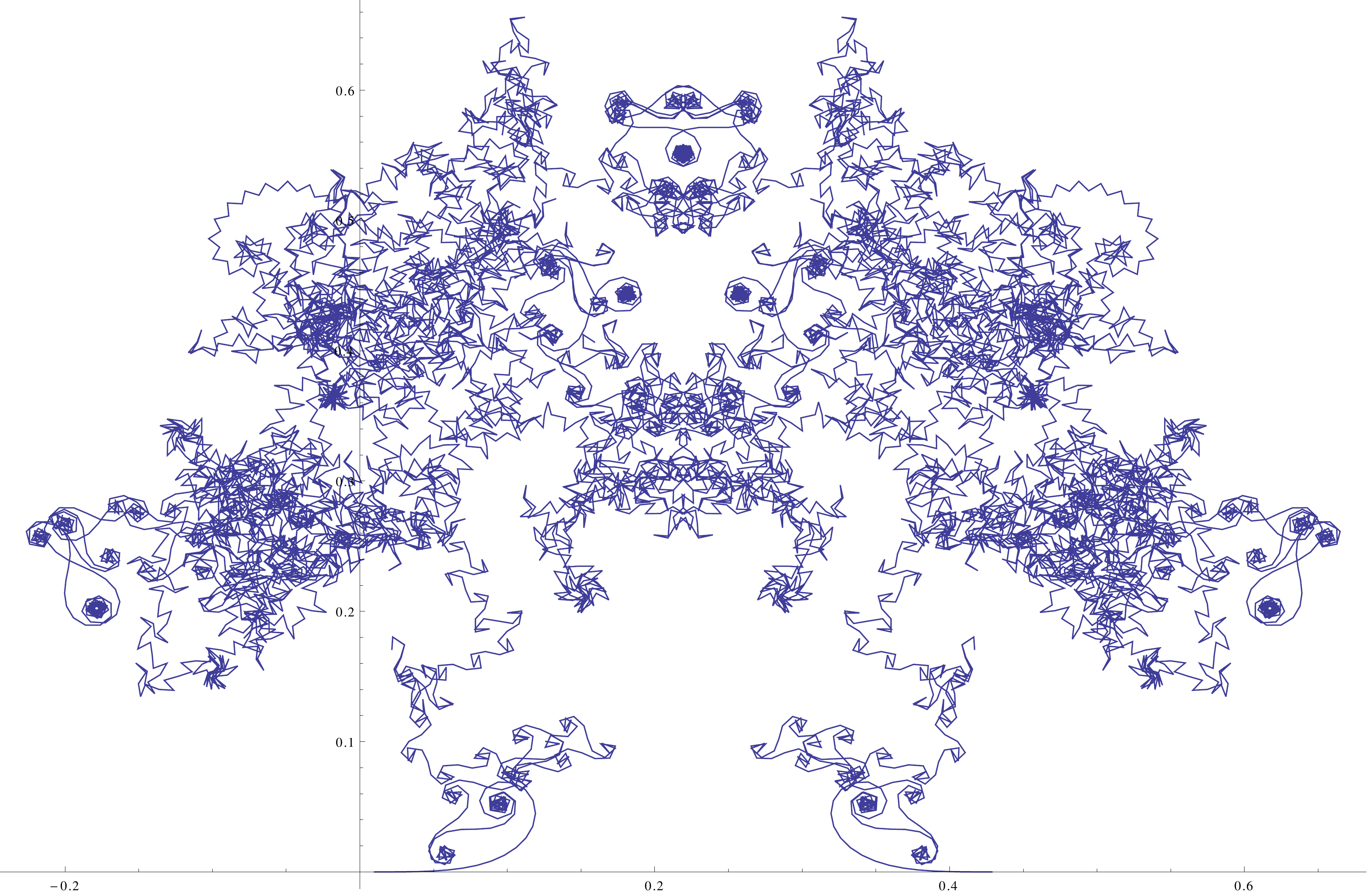}
\end{figure}
\par
The complete character sums in this case are given by
$$
\birch_p(a)=\frac{1}{\sqrt{p}}\sum_{x\in\Fp}\psi_p(ax+x^3).
$$
\par
These sums were considered by Birch~\cite[\S 3]{birch}, who
conjectured that they are Sato-Tate-distributed (on average over
$a$). This statement was first proved by
Livn\'e~\cite{livne}. However, his result (and its proof) would not
suffice for our applications, because the argument lacks the
group-theoretic interpretation of the Sato-Tate distribution.  We rely
crucially on Deligne's equidistribution theorem that provides it (as a
very special case), as well as on algebraic properties of the
monodromy group.

%
\end{remark}

For Kloosterman sums, as we mentioned, we can also use average over
the additive character to get convergence in law:

\begin{theorem}\label{th-main3}
For $p$ prime, $(\alpha,a)\in\Fpt\times\Fpt$, let
$$
t\mapsto \ekp_p(t;\alpha,a)
$$
denote the continuous function interpolating the partial sums
$$
\frac{1}{\sqrt{p}}\sum_{1\leq x\leq j}\psi_p(\alpha(ax+\bar{x})).
$$
\par
The processes $(\ekp_p(t))$ on the probability spaces
$$
(\Fpt\times\Fpt,\text{uniform probability})
$$
converge to the process $(\kpath(t))_{t\in [0,1]}$ in the sense of
convergence in law in $C([0,1])$.
\end{theorem}


We recall (see, e.g.,~\cite[Ch. 1, \S 1]{billingsley}
or~\cite[Def. 0.5.5]{revuz-yor}) that the convergence in law to
$\kpath$ of any sequence of $C([0,1])$-valued processes
$(L_p(t))_{t\in [0,1]}$ in $C([0,1])$ means that for any map
$$
\varphi\,:\, C([0,1])\lra \Cc
$$
which is continuous and bounded on $C([0,1])$, with respect to the
topology of uniform convergence, we have
$$
\lim_{p\ra +\infty}\expect(\varphi(L_p))= \expect(\varphi(\kpath)).
$$
\par
This condition is stronger than the convergence of finite-dimensional
distributions. Indeed, given the convergence of finite-dimensional
distributions of $(L_p(t))$ to those of $(\kpath(t))$, one knows that
convergence in law is equivalent to the weak-compactness property
known as \emph{tightness} (this is Prokhorov's Theorem, see
e.g.~\cite[Th. 7.1]{billingsley}).  
\par
As an application of Theorems~\ref{th-main2} and~\ref{th-main3}, we
will obtain fairly sharp bounds for the probability of large values of
partial sums of the corresponding families of exponential sums:

\begin{theorem}\label{th-tails}
  There exists $c>0$ such that we have
\begin{multline*}
  c^{-1}\exp(-\exp(cA)) \\\leq \lim_{p\ra +\infty} \frac{1}{p-1}
  \Bigl|\Bigl\{a\in\Fpt\,\mid\,
\max_{0\leq j\leq p-1}
\frac{1}{\sqrt{p}} \Bigl|\sum_{0\leq x\leq j}\psi_p(ax+x^3)\Bigr|
\geq A\Bigr\}\Bigr| \\\leq
  c\exp(-\exp(c^{-1}A)),
\end{multline*}
as well as
\begin{multline*}
  c^{-1}\exp(-\exp(cA)) \\\leq \lim_{p\ra +\infty} \frac{1}{(p-1)^2}
  \Bigl|\Bigl\{(\alpha,a)\in\Fpt\times\Fpt\,\mid\, \max_{1\leq j\leq
    p-1} \frac{1}{\sqrt{p}} \Bigl|\sum_{1\leq x\leq
    j}\psi_p(\alpha(ax+\bar{x}))\Bigr| \geq A\Bigr\}\Bigr| \\\leq
  c\exp(-\exp(c^{-1}A))
\end{multline*}
for all $A>0$. In particular, the two limits exist, and in fact they
are equal.
\end{theorem}

As we will clearly see, convergence of finite distributions and
tightness are valid in great generality for many processes
corresponding to partial sums of one-variable exponential sums. They
follow naturally from two important properties:
\par
(1) Computation and properties of the \emph{monodromy group} of
certain families of exponential sums (these are the Kloosterman sums
for Theorems~\ref{th-main} and~\ref{th-main3}, and the Birch sums for
Theorem~\ref{th-main2}, but other cases lead to slightly different
sums). This is used to prove convergence of finite distributions.
\par
(2) Existence of estimates (in a suitable averaged form) of
\emph{short} sums of the original summands, over (arbitrarily located)
intervals of length $y$ close to $\sqrt{p}$ in logarithmic scale; this
is used to prove tightness. Such sums are at the edge of the so-called
\emph{P\'olya-Vinogradov range}, which refers to $y$ a bit larger than
$\sqrt{p}\log p$, and which can be treated in considerable generality.
\par
The first ingredient is exceptionally deep: it involves all of
Deligne's work on the Riemann Hypothesis over finite
fields~\cite{weilii}, as well as many additional algebraic and
geometric results.  The second ingredient is also very delicate, and
is not currently known in great generality, although one can certainly
conjecture that it should be true under very general conditions. (It
is, at the current time, not known for Kloosterman sums when averaging
only over $a$, which is the reason that Theorem~\ref{th-main3}
requires additional averaging over the additive character).
\par
It is a very appealing and striking feature of this work that it shows
how these two arithmetic aspects of exponential sums are unified to
contribute to a single clear conclusion.  Since both properties are
very important in many applications in different ways, this is quite
an interesting phenomenon. In fact, some sporadic relations between
these two types of properties have already appeared (e.g., in the
proof of Burgess's bounds for short character sums, see
e.g.~\cite[Th. 12.6]{ant}, or in more recent work of Fouvry and
Michel~\cite{fouvry-michel} on exponential sums over primes). None is,
as far as we know, as direct and clearly focused as the phenomenon
that we present.
\par
The outline of the remainder of the paper is as follows: in
Section~\ref{sec-finite}, we present the proof of
Theorem~\ref{th-main}, and more generally of convergence of finite
distributions for the situation of Theorem~\ref{th-main2} (and
potentially many other cases, possibly for a different random Fourier
series than $\kpath(t)$). In Section~\ref{sec-tight}, we address the
additional condition of tightness, and relate it to short exponential
sums. In Section~\ref{sec-applis}, we give applications, especially
proving Theorem~\ref{th-tails}. Finally, in Section~\ref{sec-other} we
make a few remarks concerning other potential cases of convergence in
law of paths of exponential sums, as well as concerning a few earlier
works that have some similarity with this paper (besides those of
Loxton already mentioned).

\subsection*{Acknowledgment.} Thanks to J. Bober, L. Goldmakher,
A. Granville and D. Koukoulopoulos for comments concerning their
work~\cite{bggk}.

\subsection*{Notation.} For a prime $p$ and a function $\varphi\,:\,
\Fp\lra \Cc$, we denote by
$$
\hat{\varphi}(h)=\frac{1}{\sqrt{p}}\sum_{x\in \Fp}\varphi(x)\psi_p(hx)
$$
the unitarily normalized Fourier transform modulo $p$. We then have
the inversion formula
$$
\varphi(x)=\frac{1}{\sqrt{p}}\sum_{h\in\Fp}\hat{\varphi}(x)\psi_p(-hx)
$$
and the Plancherel identity
$$
\sum_{x\in\Fp}\varphi_1(x)\overline{\varphi_2(x)}=
\sum_{h\in\Fp}\hat{\varphi}_1(h)\overline{\hat{\varphi}_2(h)}.
$$
\par
For any probability space $(\Omega,\Sigma,\proba)$, we denote by
$\proba(A)$ the probability of some event and by $\expect(X)$ the
expectation of a (complex-valued) random variable. We sometimes use
different probability spaces, but keep the same notation for all
expectations and probabilities.
\par
For $\sigma>0$, a $\sigma$-subgaussian (real-valued) random variable
$N$ is a random variable such that
$$
\expect(e^{\lambda N})\leq e^{\sigma^2\lambda^2/2}
$$
for all $\lambda\in\Rr$. We then have
$$
\expect(|N|^k)\leq c_k\sigma^{k}
$$
for any $k\geq 0$, where $c_k=k2^{k/2}\Gamma(k/2)$, in particular
$c_4=16$.
\par
For $\sigma>0$, it will be convenient to say that a complex-valued
random variable $N=R+iI$ is $\sigma$-subgaussian, with $R$ and $I$
real-valued, if $R$ and $I$ are $\sigma$-subgaussian (we make no
assumption on the independence or not of $R$ and $I$). We then have
\begin{equation}\label{eq-subg}
  \expect(|N|^k)\leq c'_k\sigma^{k}
\end{equation}
for some $c'_k>0$. Here one can for instance take $c'_4=256$. If
$N_1$, $N_2$ are $\sigma_i$-subgaussian and independent (real or
complex-valued), then $N_1+N_2$ is
$\sqrt{\sigma_1^2+\sigma_2^2}$-subgaussian.
\par
We will write $\|\varphi\|_{\infty}$ for the supremum norm of a
continuous function $\varphi$ on $[0,1]$; to avoid confusion, if $X$
is a random variable defined on a space $\Omega$, we will write
$\|X\|_{L^{\infty}(\Omega)}$ for the (essential) supremum of $X$.

\section{Proof of convergence of finite distributions}\label{sec-finite}

We begin by studying the random series $\kpath(t)$, since some of its
properties are relevant to the proof of Theorem~\ref{th-main}.

\begin{proposition}\label{pr-limit}
  For a fixed $t\in [0,1]$, the symmetric partial sums
$$
\kpath_m(t)=\sum_{|h|<m}\frac{e^{2i\pi ht}-1}{2i\pi h}\st_h
$$
of the random series defining $\kpath(t)$ converge in law, almost
surely, and in every space $L^q(\Omega)$ for $1\leq q<+\infty$, where
$\Omega$ is the probability space on which on the Sato-Tate variables
$\st_h$ are defined. In fact, we have
\begin{equation}\label{eq-linf}
\|\kpath_m(t)\|_{L^{\infty}(\Omega)}\ll (\log m),
\end{equation}
and
\begin{equation}\label{eq-lun}
\expect(|\kpath(t)-\kpath_m(t)|)\ll m^{-1/2},
\end{equation}
for $m\geq 1$, where the implied constants are absolute.
\par
For any $t\in [0,1]$, the Laplace transform
$$
\expect(e^{\lambda \Reel(\kpath(t))+\nu\Imag(\kpath(t))})
$$
is well-defined for all $(\lambda,\nu)$ non-negative. In particular,
$\kpath(t)$ has moments of all orders. 
\par
The process $(\kpath(t))_{t\in [0,1]}$ is almost surely continuous. It
is also almost surely nowhere differentiable. In fact, it is almost
surely H\"older continuous of all order $\alpha<1/2$ on $[0,1]$, and
almost surely nowhere H\"older continuous of order $1/2$.
\end{proposition}


\begin{proof}
  The convergence almost surely, hence in law, of the series for any
  fixed $t$ is an immediate consequence of Kolmogorov's $3$-series
  theorem, together with the fact that the Sato-Tate measure has mean
  $0$ and is compactly supported.
\par
The other results, however, are most easily derived as consequences of
general facts about random Fourier series, which we quote from the
work~\cite{kahane} of Kahane.
\par
We can write $\kpath(t)=t\st_0+4A(2\pi t)+4iB(2\pi t)$ where $A$ and $B$
are the random Fourier series
$$
A(t)=\sum_{h\geq 1}\frac{a_h}{\pi h}\cos(ht),\quad\quad
B(t)=\sum_{h\geq 1}\frac{b_h}{\pi h}\cos(ht-\pi/2),
$$
where
$$
a_h=\frac{\st_h-\st_{-h}}{4},\quad\quad b_h=\frac{\st_h+\st_{-h}}{4},
$$
(note that $(a_h)$ and $(b_h)$ are identically distributed since the
Sato-Tate law is symmetric).
\par
Both series are of the type considered in~\cite[Ch. 5, 7, 8]{kahane},
and especially, note that the random variables $a_h$ and $b_h$ are
$1$-subgaussian
(indeed, this is a property of any centered real random variable with
absolute value bounded by $1$).  The existence of the Laplace
transforms is then given by~\cite[\S 5.5, Th. 1]{kahane},
and we see from~\cite[\S 7.4, Th. 3]{kahane} that each of $A(t)$ and
$B(t)$ (hence also $\kpath(t)$) is $\alpha$-H\"older on $[0,1]$ if
$\alpha<1/2$.  Furthermore, it follows from~\cite[\S 8.6,
Th. 4]{kahane} that each of $A(t)$ and $B(t)$ is a.s. nowhere
$1/2$-H\"older-continuous.
\par
The bound~(\ref{eq-linf}) is clear since $|\st_h|\leq 2$. Then, for
any $q\geq 1$, the convergence of the partial sums $\kpath_m(t)$ in
$L^q(\Omega)$ follows from the convergence in $L^1(\Omega)$ implied
by~(\ref{eq-lun}). For the latter, it is enough to note that
$\kpath(t)-\kpath_m(t)$ is $\sigma_m$-subgaussian with
$$
\sigma_m^2=\sum_{|h|\geq m}\Bigl|\frac{e^{2i\pi ht}-1}{2i\pi
  h}\Bigr|^2\ll\frac{1}{m},
$$
together with the property~(\ref{eq-subg}) of subgaussian random
variables.
\end{proof}

\begin{remark}
It is interesting to contrast the result with the Fourier series
$$
W(t)=\sum_{h\in\Zz}{\frac{e^{2i\pi ht}-1}{2i\pi h}N_h}
$$
where $(N_h)_{h\in\Zz}$ is a sequence of independent standard complex
\emph{gaussian} random variables, i.e., $G_h=R_h+iI_h$ where $R_h$ and
$I_h$ are independent standard normal random variables. Then (this was
already known to Paley and Wiener) $W(t)$ is a standard complex
Brownian motion (see~\cite[Ch. 16, \S 3]{kahane}
. 
\end{remark}

We now begin the proof of the second part of
Theorem~\ref{th-main}. The argument extends immediately to the cases
considered in Theorems~\ref{th-main2} and~\ref{th-main3}, because the
main arithmetic property required is also valid then.  Thus we will
only comment briefly on this part of Theorems~\ref{th-main2}
and~\ref{th-main3} after the proof.
\par
Since Proposition~\ref{pr-limit} shows that the Laplace transforms of
the finite distributions of $(\kpath(t))$ exist, we can use the method
of moments to prove convergence in law of the finite
distributions. The next proposition therefore implies
Theorem~\ref{th-main} (2), and concludes the proof of that result.

\begin{proposition}\label{pr-finite}
  Let $k\geq 1$ be given, and $0\leq t_1<\cdots<t_k\leq 1$ be
  fixed. Fix also non-negative integers $(n_1,\ldots, n_k)$ and
  $(m_1,\ldots, m_k)$. Let
$$
M_p=\frac{1}{p-1}\sum_{a\in \Fpt}
\prod_{i=1}^kK_p(t_i,a)^{n_i}\overline{K_p(t_i,a)}^{m_i}.
$$
\par
We have
$$
M_p=\expect\Bigl(
\prod_{i=1}^k\kpath(t_i)^{n_i}\overline{\kpath(t_i)}^{m_i}
\Bigr)+O(p^{-1/2}(\log p)^{m+n})
$$
for $p\geq 2$, where $n=n_1+\cdots+n_k$, $m=m_1+\cdots+m_k$, and the
implied constant depends only on $m$ and $n$.
\end{proposition}

Since the notation may obscure the essential arithmetical point, the
reader is encouraged to first read through the proof under the
assumption that $k=1$.
\par
We fix once and for all the sequence $(\st_h)_{h\in\Zz}$ of
independent Sato-Tate random variables used to define the process
$\kpath(t)$.

\begin{proof}
First of all, we deal with the linear interpolation involved in the
definition of $K_p(t)$. Let
$$
\tkp_p(t,a)=\frac{1}{\sqrt{p}}\sum_{1\leq x\leq
  (p-1)t}\psi_p(ax+\bar{x})
$$
for $p$ prime, $a\in\Fpt$ and $t\in [0,1]$. This is a discontinuous
function of $t$, and we have
\begin{equation}\label{eq-desmooth}
|K_p(t,a)-\tkp_p(t,a)|\leq \frac{1}{\sqrt{p}}
\end{equation}
for all $p$, $a$ and $t$. Moreover, from the discrete Plancherel
formula (i.e., the completion method), we have
\begin{equation}\label{eq-plancherel}
  \tkp_p(t,a)=\frac{1}{\sqrt{p}}\sum_{|h|<p/2}\alpha_p(h;t)\hypk_p(a-h)
\end{equation}
for any $a\in\Fpt$ and $t\in[0,1]$, where
$$
\alpha_p(h;t)=\frac{1}{\sqrt{p}}\sum_{1\leq x\leq (p-1)t} \psi_p(hx)
$$
are the discrete Fourier coefficients of the characteristic function
of the interval $1\leq x\leq (p-1)t$ modulo $p$. It is well-known that
\begin{equation}\label{eq-completion}
\sum_{|h|<p/2}|\alpha_p(h;t)|\leq \sqrt{p}(\log 3p),
\end{equation}
for all $t\in [0,1]$, so that in particular we have
\begin{equation}\label{eq-complete}
|\tkp_p(t,a)|\leq 2(\log 3p)
\end{equation}
for all $p$, $t$ and $a$, by Weil's bound for Kloosterman sums.
\par
We deduce from this that
$$
\Bigl|\prod_{i=1}^kK_p(t_i,a)^{n_i}\overline{K_p(t_i,a)}^{m_i}-
\prod_{i=1}^k\tkp_p(t_i,a)^{n_i}\overline{\tkp_p(t_i,a)}^{m_i}\Bigr|
\ll p^{-1/2}(\log p)^{m+n},
$$
where
the implied constant depends only on $m$ and $n$. Hence it is
enough to prove the moment estimate for
$$
\tilde{M}_p=\frac{1}{p-1}\sum_{a\in \Fpt}
\prod_{i=1}^k\tkp_p(t_i,a)^{n_i}\overline{\tkp_p(t_i,a)}^{m_i}.
$$
\par
We compute $\tilde{M}_p$ by replacing each $\tkp_p(t_i,a)$ and its
conjugate by the formula~(\ref{eq-plancherel}) and taking the $n_i$ or
$m_i$-th power.  We obtain
\begin{multline*}
  \tilde{M}_p=\frac{1}{p^{(m+n)/2}}\frac{1}{p-1}\sum_{a\in \Fpt}
  \multsum_{\uple{h}_1,\ldots,\uple{h}_k}\alpha_p(\uple{h}_1;t_1)\cdots
  \alpha_p(\uple{h}_k;t_k)\\
  \times \prod_{l=1}^{n_1+m_1}\hypk_p(a-h_{1,l})\cdots
  \prod_{l=1}^{n_k+m_k}\hypk_p(a-h_{k,l}),
\end{multline*}
where each
$$
\uple{h}_j=(h_{j,1},\ldots, h_{j,n_j},h_{j,n_j+1},\ldots,
h_{j,n_j+m_j})
$$
ranges over all $(n_j+m_j)$-tuples of integers $h_{j,l}$ in
$]-p/2,p/2[$ and
$$
 \alpha_p(\uple{h}_j;t_j)=
\prod_{l=1}^{n_j}\alpha_p(h_{j,l};t_j)
\prod_{l=1}^{m_j}\overline{\alpha_p(h_{j,n_j+l};t_j)}.
$$
\par
Exchanging the order of the sums, we deduce
$$
\tilde{M}_p=\frac{1}{p^{(m+n)/2}}\multsum_{\uple{h}_1,\ldots,
  \uple{h}_k} \prod_{1\leq j\leq
  k}\alpha_p(\uple{h}_j;t_j)S(\uple{h}_1,\cdots,\uple{h}_k;p)
$$
with
$$
S(\uple{h}_1,\ldots,\uple{h}_k;p) =\frac{1}{p-1} \sum_{a\in\Fpt}
\prod_{l=1}^{n_1+m_1}\hypk_p(a-h_{1,l})\cdots
\prod_{l=1}^{n_k+m_k}\hypk_p(a-h_{k,l}).
$$
\par
The sums $S(\uple{h}_1,\ldots,\uple{h}_k;p)$ are complete sums of
products of Kloosterman sums. The crucial point, which we explain
below in Lemma~\ref{lm-katz}, is that from Deligne's Riemann
Hypothesis over finite fields, the computation of the geometric
monodromy group of the Kloosterman sheaf of rank $2$ by
Katz~\cite{katz-gkm}, and the Goursat-Kolchin-Ribet criterion~\cite[\S
1.8]{katz-esde}, we can derive the estimate
$$
S(\uple{h}_1,\ldots,\uple{h}_k;p)=
\expect\Bigl(\prod_{l=1}^{n_1+m_1}\st_{h_{1,l}} \cdots
\prod_{l=1}^{n_k+m_k}\st_{h_{k,l}}\Bigr)+O(p^{-1/2})
$$
where the implied constant depends only on $m$ and $n$.
\par
By~(\ref{eq-completion}), the contribution $E_p$ of the error terms to
$\tilde{M}_p$ is bounded by
$$
E_p \ll p^{-1/2}\frac{1}{p^{(m+n)/2}}p^{(m+n)/2}(\log 3p)^{m+n} \ll
p^{-1/2}(\log 3p)^{m+n}
$$
where the implied constant depends only on $m$ and $n$.
\par
On the other hand, by reverting the computation, we get
$$
\tilde{M}_p=
\expect\Bigl( \prod_{i=1}^kX_p(t_i)^{n_i}\overline{X_p(t_i)}^{m_i}
\Bigr)+
O(p^{-1/2}(\log p)^{m+n}).
$$
where the random variables $X_p(t)$ are given by
\begin{equation}\label{eq-xp}
X_p(t)=\sum_{|h|<p/2} \frac{\alpha_p(h;t)}{\sqrt{p}}\st_h.
\end{equation}
\par
We now denote
$$
\beta(h;t)=\frac{e^{2i\pi ht}-1}{2i\pi h}
$$
for $h\in\Zz$, with $\beta(0;t)=t$, and we consider the partial sums
$$
\kpath_p(t)=\sum_{|h|<p/2} \beta(h;t)\st_h
$$
of $\kpath(t)$. From~(\ref{eq-linf}) and~(\ref{eq-lun}) in
Proposition~\ref{pr-limit}, we see that
$$
\expect\Bigl(
\prod_{i=1}^k\kpath_p(t_i)^{n_i}\overline{\kpath_p(t_i)}^{m_i}
\Bigr)=\expect\Bigl(
\prod_{i=1}^k\kpath(t_i)^{n_i}\overline{\kpath(t_i)}^{m_i} \Bigr)+
O(p^{-1/2}(\log p)^{m+n}),
$$
where the implied constant depends only on $m$ and $n$.
\par
It is therefore enough to prove that
$$
\expect\Bigl( \prod_{i=1}^kX_p(t_i)^{n_i}\overline{X_p(t_i)}^{m_i}
\Bigr)= \expect\Bigl(
\prod_{i=1}^k\kpath_p(t_i)^{n_i}\overline{\kpath_p(t_i)}^{m_i} \Bigr)
+O(p^{-1/2}(\log p)^{m+n})
$$
in order to finish the proof of the proposition.
\par
In view of the bound~(\ref{eq-linf}) and the analogue
$$
\|X_p(t)\|_{L^{\infty}(\Omega)}\ll (\log p)
$$
where the implied constant is absolute, 
it suffices to prove that
$$
\expect(|X_p(t)-\kpath_p(t)|^2)\ll p^{-1}.
$$
\par
But since the random variables $\st_h$ are independent with
$\expect(\st_h)=0$ and $\expect(|\st_h|^2)=1$, we have
$$
\expect(|X_p(t)-\kpath_p(t)|^2)= \sum_{|h|<p/2}
\Bigl|\frac{\alpha_p(h;t)}{\sqrt{p}}-\beta(h;t)\Bigr|^2.
$$
\par
By definition and summing a geometric sum, we get
$$
\frac{\alpha_p(h;t)}{\sqrt{p}}=\frac{1}{p}\sum_{1\leq x\leq
  (p-1)t}\psi_p(hx) = \frac{\psi_p(h)}{p}\frac{1-\psi_p(h\lfloor
  (p-1)t\rfloor)}{1-\psi_p(h)},
$$
with the convention that 
$$
\frac{\alpha_p(0;t)}{\sqrt{p}}= \frac{\lfloor (p-1)t\rfloor }{p}.
$$
\par
For $h=0$, we therefore find that
$$
\frac{\alpha_p(0;t)}{\sqrt{p}}-\beta(0;t)= \frac{\lfloor
  (p-1)t\rfloor}{p}-t\ll \frac{1}{p}
$$
for all $t\in [0,1]$ and $p\geq 3$.
\par
For all $h$ such that $1\leq |h|<p/2$, we can write for instance
\begin{multline*}
  \frac{\alpha_p(h;t)}{\sqrt{p}}-\beta(h;t) = \frac{(\psi_p(\lfloor
    (p-1)t\rfloor h)-1)-(e(ht)-1)}{2i\pi h} \\-\Bigl(\psi_p(\lfloor
  (p-1)t\rfloor h)-1\Bigr)\Bigl(\frac{1}{2i\pi
    h}-\frac{1}{p(\psi_p(h)-1)}\Bigr)\\+(\psi_p(h)-1)\frac{\psi_p(\lfloor
    (p-1)t\rfloor h)-1}{p(\psi_p(h)-1)},
\end{multline*}
and then simple bounds for the three terms show that we also have
$$
\frac{\alpha_p(h;t)}{\sqrt{p}}-\beta(h;t) \ll \frac{1}{p}
$$
uniformly for all $t\in [0,1]$ and all $p$.  
\par
Squaring and summing over $h$, it follows therefore that
$$
\expect(|X_p(t)-\kpath_p(t)|^2)\ll p^{-1},
$$
which gives  the desired bound and finishes the proof.
\end{proof}

Here is the crucial arithmetic lemma that we used:

\begin{lemma}
\label{lm-katz}
With notation as in the proof, we have
$$
S(\uple{h}_1,\ldots,\uple{h}_k;p)=
\expect\Bigl(\prod_{l=1}^{n_1+m_1}Z_{h_{1,l}} \cdots
\prod_{l=1}^{n_k+m_k}Z_{h_{k,l}}\Bigr)+O(p^{-1/2})
$$
where the $(Z_h)_{h\in \Fp}$ are independent random variables with
Sato-Tate distributions, and the implied constant depends only on $m$
and $n$.
\end{lemma}

\begin{proof}
We can write
$$
S(\uple{h}_1,\ldots,\uple{h}_k;p)=
\frac{1}{p-1}\sum_{a\in\Fpt}
\prod_{\tau\in \Fp}\hypk_p(a+\tau)^{\mu(\tau)}
$$
where
$$
\mu(\tau)=\sum_{j=1}^k|\{1\leq l\leq n_j+m_j\,\mid\,
h_{j,l}=\tau\mods{p}\}|,
$$
for any $\tau\in\Fp$, is the multiplicity of the factor
$\hypk_p(a+\tau)$ among the shifted Kloosterman sums in
$S(\uple{h}_1,\ldots,\uple{h}_k;p)$.
\par
This type of sums of products can be estimated by the Riemann
Hypothesis over finite fields, as explained in detail
in~\cite{sums-of-products}. More precisely, Katz
showed~\cite[Th. 11.1]{katz-gkm} that the geometric and arithmetic
monodromy groups of the Kloosterman sheaf $\sheaf{F}=\HYPK_2$ are
equal and isomorphic to $\SL_2$. Furthermore, if $\tau\not=0$ there
does not exist a rank $1$ sheaf $\sheaf{L}$ such that
$$
[+\tau]^*\HYPK_2\simeq \HYPK_2\otimes\sheaf{L}
$$
(most simply seen here because the left-hand side is unramified at
$0$, while the right-hand side is ramified). Using the
Goursat-Kolchin-Ribet criterion~\cite[\S 1.8]{katz-esde}, it follows
that the geometric and arithmetic monodromy groups of
$$
\bigoplus_{\tau\in\Fp}[+\tau]^*\sheaf{F}
$$ 
are equal to $\SL_2\times\cdots\times \SL_2$. The Riemann Hypothesis
then gives the asymptotic formula
$$
S(\uple{h}_1,\ldots,\uple{h}_k;p)= \prod_{\tau\in\Fp}A(\mu(\tau))
+O(p^{-1/2})
$$
where $A(\mu)$, for any integer $\mu\geq 0$, denotes the multiplicity
of the trivial representation of $\SU_2$ in the $\mu$-th tensor power
of its standard $2$-dimensional representation, and the implied
constant depends only on
$$
\sum_{\tau\in\Fp}{\mu(\tau)}=\sum_{i}(n_i+m_i)=m+n
$$
(see also~\cite[Cor. 3.3]{sums-of-products} for this statement).
\par
However, we have by character theory
$$
A(\mu)=\expect(\st^{\mu})
$$
for any Sato-Tate distributed random-variable $\st$ and $\mu\geq
0$. Thus by reversing the computation, we see that
$$
\prod_{\tau\in\Fp}A(\mu(\tau))=
\expect\Bigl(\prod_{l=1}^{n_1+m_1}Z_{h_{1,l}} \cdots
\prod_{l=1}^{n_k+m_k}Z_{h_{k,l}}\Bigr)
$$
where the $Z_{h}$ are independent and Sato-Tate distributed.
\end{proof}

\begin{remark}
(1) We emphasize once again that, for our application, it is essential
to obtain the correct main term, and not only a criterion for
cancellation in these sums.  This contrasts with many other
applications of such estimates.
\par
(2) Interestingly, similar sums of products of shifted Kloosterman
sums also occurred in a recent work of Irving~\cite{irving} concerning
the divisor function in arithmetic progressions to smooth moduli;
there, however, only the cancellation criterion was required.
\end{remark}

We can now see why convergence of finite distributions also holds in
the case considered in Theorem~\ref{th-main2}. We have then
$$
\frac{1}{\sqrt{p}}\sum_{0\leq x\leq pt}\psi_p(ax+x^3)
=\frac{1}{\sqrt{p}}\sum_{|h|<p/2}\alpha'_p(h;t)\birch_p(a-h)
$$
exactly as in~(\ref{eq-plancherel}), where $\alpha'_p(h;t)$ are the
discrete Fourier coefficients of the interval $0\leq n\leq pt$ modulo
$p$ and
$$
\birch_p(a)=\frac{1}{\sqrt{p}}\sum_{x\in\Fp}{\psi_p(ax+x^3)}
$$
are the Birch sums. Since Katz also showed that the geometric and
arithmetic monodromy groups of the lisse sheaf $\sheaf{G}$ on
$\Aa^1_{\Fp}$ parameterizing these sums (namely, the sheaf-theoretic
Fourier transform of the Artin-Schreier sheaf
$\sheaf{L}_{\psi_p(x^3)}$) are both equal to $\SL_2$ for $p>7$
(see~\cite[Th. 19, Cor. 20]{katz-mg}), the proof of
Proposition~\ref{pr-finite} applies essentially verbatim to give
convergence of finite distributions. In checking the analogue of
Lemma~\ref{lm-katz}, one has to check that the sheaf $\sheaf{G}$ is
such that there is no geometric isomorphism
$$
[+\tau]^*\sheaf{G}\simeq \sheaf{G}\otimes\sheaf{L}
$$
where $\sheaf{L}$ is of rank $1$ and $\tau\not=0$. But indeed, such a
sheaf $\sheaf{L}$ would need to be lisse on $\Aa^1$ (since $\sheaf{G}$
is), and therefore trivial. Then the condition
$$
[+\tau]^*\sheaf{G}\simeq \sheaf{G}
$$
would imply (by taking Fourier transform) that
$\sheaf{L}_{\psi(X^3)}\simeq \sheaf{L}_{\psi(X^3+\tau X)}$, which is
not the case for $\tau\not=0$.
\par
For the process of Theorem~\ref{th-main3}, proving convergence in
finite distributions is only a matter of checking that the convergence
in finite distribution for Kloosterman sums holds for any choice of
non-trivial additive character modulo $p$, instead of $\psi_p$, and
this is immediate.


\section{Proof of tightness}\label{sec-tight}

Now we consider tightness to finish the proof of
Theorem~\ref{th-main2}. More generally, we consider a sequence of
processes $(K^{\mathcal{X}}_p(t))_{t\in [0,1]}$ constructed as
described before Theorem~\ref{th-main2}, with summands
$\xi_p(x,\omega)$ defined either for $x\in \Fp$ or $x\in\Fpt$.
\par
We will use Kolmogorov's criterion to find a condition that implies
tightness:

\begin{proposition}[Kolmogorov tightness criterion]
Let $(L_p(t))_{t\in [0,1]}$ be a sequence of $C([0,1])$-valued
processes such that $L_p(0)=0$ for all $p$. 
\par
If there exist constants $\alpha>0$, $\delta>0$ and $C\geq 0$, such
that for any $p$ and any $s<t$ in $[0,1]$, we have
\begin{equation}\label{eq-kol}
  \expect(|L_p(t)-L_p(s)|^{\alpha})\leq
  C|t-s|^{1+\delta}
\end{equation}
then the sequence $(L_p(t))$ is tight.
\end{proposition}

This is found in, e.g.,~\cite[Th. XIII.1.8]{revuz-yor}.
We then obtain the following criterion for paths of exponential sums: 

\begin{lemma}[Tightness and short sums]\label{lm-tight}
  Assume that $\mathcal{X}=(\xi_p(x))_{x\in\Fp}$ is defined on
  $\Omega_p$, a finite set with uniform probability measure, and
  satisfies the following conditions:
\par
\emph{(1)} There exists $H\geq 1$ such that we have
$$
|\xi_p(x,\omega)|\leq 1,\quad\quad |\hat{\xi}_p(h,\omega)|\leq H
$$
for all primes $p$, $x\in \Fp$, $h\in\Fp$ and $\omega\in\Omega_p$,
where
$$
\hat{\xi}_p(\omega,h)=
\frac{1}{\sqrt{p}}
\sum_{x\in\Fp}\xi_p(x,\omega)\psi_p(hx)
$$
is the discrete Fourier transform of $x\mapsto \xi_p(x,\omega)$.
\par
\emph{(2)} We have
$$
\frac{1}{|\Omega_p|} \sum_{\omega\in\Omega_p} \hat{\xi}_p(h_1,\omega)
\hat{\xi}_p(h_2,\omega) \overline{\hat{\xi}_p(h_3,\omega)
  \hat{\xi}_p(h_4,\omega)} =\expect(\st_{h_1}\cdots
\st_{h_4})+O(p^{-1/2})
$$
for all primes $p$ and $(h_1,\ldots,h_4)\in\Ff_p^4$.
\par
\emph{(3)} There exist $\alpha>0$, $\delta_1>0$ and $\delta_2>0$ such
that, for any prime $p$, any interval $I\subset \Fpt$ of length
$$
p^{1/2-\delta_1}\leq |I|\leq p^{1/2+\delta_1},
$$
we have
\begin{equation}\label{eq-short}
  \frac{1}{|\Omega_p|}\sum_{\omega\in\Omega_p}
  \Bigl|\frac{1}{\sqrt{p}}\sum_{x\in I}\xi_p(x,\omega)\Bigr|^{\alpha} \ll
  p^{-1/2-\delta_2}.
\end{equation}
\par
Then the sequence $(K_p^{\mathcal{X}}(t))_{t\in [0,1]}$ is tight as
$C([0,1])$-valued random variables. Moreover, the same holds if the
summands $\mathcal{X}=(\xi_p(x))_{x\in\Fpt}$ are parameterized by
$\Fpt$ instead of $\Fp$.
\end{lemma}

\begin{remark}
  Note that (2) is, in practice, a special case of the main estimate
  of (the analogue of) Lemma~\ref{lm-katz} that is used to prove
  convergence in finite distributions. Moreover, (1) is a standard
  condition for the type of exponential sums we consider (typically,
  bounds on the Fourier transform would already follow from Weil's
  theory of exponential sums in one variable). 
\par
Thus, the practical meaning of this lemma is that, once convergence of
finite distributions is known ``for standard reasons'', tightness
becomes a consequence of the estimate~(\ref{eq-short}). The latter
concerns the average distribution (over $\omega\in\Omega_p$) of short
partial sums of the summands $\xi_p(x,\omega)$, where the length of
the sums is close to $p^{1/2}$, but can be a bit smaller.
\par
If, as one certainly expects in many cases, there exists $\eta>0$ such
that we have a uniform non-trivial individual bound
$$
\frac{1}{\sqrt{p}}\sum_{x\in I}\xi_p(x,\omega)\ll p^{-\eta}
$$
for all $\omega\in\Omega_p$ and all intervals $I$ of length about
$p^{1/2}$ (as in the statement of (3)), then taking $\alpha>0$ large
enough yields (3).
\par
This will suffice for Birch sums, but is not known for Kloosterman
sums at this time.  However, in some cases, one can get average bounds
without proving first individual estimates, and an example is given by
Theorem~\ref{th-main3}.
\end{remark}

We first assume the validity of Lemma~\ref{lm-tight}, and use it to
prove Theorems~\ref{th-main2} and~\ref{th-main3}.

\begin{proof}[Proof of Theorem~\ref{th-main2}]
  Recall that $\Omega_p=\Fpt$ and $\xi_p(x,a)=\psi_p(ax+x^3)$.  The
  first condition of Lemma~\ref{lm-tight} is then clear.  The second
  condition holds by (the analogue for the Birch sums of)
  Lemma~\ref{lm-katz}. For (3), the point is that individual bounds
  for sums over intervals of polynomials of rather short length are
  known, from methods such as Weyl differencing, so we can use the
  argument indicated in the previous remark.
\par
Precisely, by Weyl's method, one gets
$$
\sum_{x\in I}\psi_p(ax+x^3)\ll
|I|^{1+\eps}\Bigl(\frac{1}{|I|}+\frac{p}{|I|^3}\Bigr)^{1/4} \ll
|I|^{1/4+\eps}p^{1/4},
$$
for $1\leq |I|<p$ and for any $\eps>0$, where the implied constant
depends only on $\eps$ (see, e.g.,~\cite[Lemma 20.3]{ant}). In
particular, if we assume that
$$
p^{5/12}\leq |I|\leq p^{7/12}
$$
(for instance), then we have
$$
\frac{1}{\sqrt{p}}\sum_{x\in I}\psi_p(ax+x^3)\ll p^{-\eta}
$$
where $\eta>0$ and the implied constant are absolute.  For any
$\alpha>0$, it follows that
$$
\frac{1}{p-1}\sum_{a\in\Fpt} \Bigl|\frac{1}{\sqrt{p}} \sum_{x\in
  I}\psi_p(ax+x^3)\Bigr|^{\alpha} \ll p^{-\alpha\eta},
$$
and selecting $\alpha$ large enough, we obtain the desired
estimate~(\ref{eq-short}).
\end{proof}

\begin{proof}[Proof of Theorem~\ref{th-main3}]
  As before, it only remains to prove~(\ref{eq-short}) for the process
  $(\ekp_p(t))$ (here the summands $\psi_p(\alpha(ax+\bar{x}))$ are
  parameterized by $x\in\Fpt$). We take $\alpha=4$ and compute the
  fourth moment (just as Kloosterman did for the full interval to get
  the first non-trivial bounds for Kloosterman sums). 
\par
We have
\begin{multline*}
  \frac{1}{(p-1)^2}\sum_{(\alpha,a)\in\Fpt\times\Fpt}
  \Bigl|\frac{1}{\sqrt{p}}\sum_{x\in
    I}\psi_p(\alpha(ax+\bar{x}))\Bigr|^{4}
  \\=\frac{1}{p^2}\sumsum_{x_1,\ldots,x_4\in I}
  \frac{1}{p-1}\sum_{\alpha\in\Fpt}\psi_p\Bigl( \alpha\Bigl(
  \frac{1}{x_1} +\frac{1}{x_2} -\frac{1}{x_3} -\frac{1}{x_4} \Bigr)
  \Bigr)\\
  \times \frac{1}{p-1}\sum_{a\in\Fpt} \psi_p(\alpha
  a(x_1+x_2-x_3-x_4)).
\end{multline*}
\par
By orthogonality of characters, this is equal to
$$
\frac{1}{p(p-1)}\sumsum_{\substack{x_1,\ldots,x_4\in I\\x_1+x_2=x_3+x_4}}
  \frac{1}{p-1}\sum_{\alpha\in\Fpt}\psi_p\Bigl( \alpha\Bigl(
  \frac{1}{x_1} +\frac{1}{x_2} -\frac{1}{x_3} -\frac{1}{x_4} \Bigr)
  \Bigr)
+O(|I|^4p^{-3})
$$
and then to
$$
\frac{1}{(p-1)^2}\sumsum_{\substack{x_1,\ldots,x_4\in
    I\\x_1+x_2=x_3+x_4\\ x_1^{-1}+x_2^{-1}=x_3^{-1}+x_4^{-1}}}1
+O(|I|^4p^{-3}).
$$
\par
But, for any fixed $x_1$ and $x_2$, provided $x_1+x_2\not=0$,
the equations
$$
\begin{cases}
x_3+x_4=x_1+x_2\\
x_3^{-1}+x_4^{-1}=x_1^{-1}+x_2^{-1}
\end{cases}
$$
have at most two pairs of solutions $(x_3,x_4)$, so that the
contribution of these $(x_1,x_2)$ is at most $2|I|^2/(p-1)^2$. On the
other hand, if $x_1+x_2=0$, then we also have $x_3+x_4=0$, so that
these contribute also at most $|I|^2(p-1)^{-2}$. Hence we get
$$
  \frac{1}{(p-1)^2}\sum_{(\alpha,a)\in\Fpt\times\Fpt}
  \Bigl|\frac{1}{\sqrt{p}}\sum_{x\in
    I}\psi_p(\alpha(ax+\bar{x}))\Bigr|^{4}
\ll |I|^2p^{-2}+|I|^4p^{-3}.
$$
\par
If we take $|I|$ close to $p^{1/2}$, this is close to $p^{-1}$,
and~(\ref{eq-short}) therefore follows easily. 
\end{proof}

\begin{remark}
(1)  Without the average over $\alpha$, we obtain
\begin{multline*}
\frac{1}{p-1}\sum_{a\in\Fpt}
  \Bigl|\frac{1}{\sqrt{p}}\sum_{x\in I}
\psi_p(ax+\bar{x})\Bigr|^{4}
=\frac{1}{p(p-1)}\sumsum_{\substack{x_1,\ldots,x_4\in I\\x_1+x_2=x_3+x_4}}
\psi_p\Bigl(  \frac{1}{x_1} +\frac{1}{x_2} -\frac{1}{x_3} -\frac{1}{x_4} \Bigr)
\\  +O(|I|^4p^{-3}).
\end{multline*}
\par
Since the number of points of summation is about $|I|^3$ (because $I$
is an interval), which leads to a bound $p^{-1/2}$ when $|I|$ is
itself $p^{1/2}$, the difficulty is therefore that we must get some
cancellation in the exponential sum over the $x_i$'s.
\par
(2) Interestingly, if we interpolate between the partial sums
$$
\sum_{1\leq x\leq j}\psi_p(x+a\bar{x})
$$
(moving the parameter $a$), although the endpoint is still the
Kloosterman sum $\hypk_p(a)$, the corresponding process \emph{does}
converge in $C([0,1])$, but with the slightly different limit
$$
\sum_{h\not=-1}\frac{e(ht)-1}{2i\pi
  h}\st_h=\kpath(t)+\frac{e(-t)-1}{2i\pi}\st_{-1}
$$
(the value $h=-1$ is omitted because the relevant analogue of
Lemma~\ref{lm-katz} involves a product over $h$ of $\hypk_p(a(h+1))$,
which is constant, equal to $1/\sqrt{p}$, for $h=-1$, so that any
moment where $h=-1$ appears with positive multiplicity does not
contribute to the asymptotic).  Here tightness follows because
orthogonality gives
\begin{multline*}
  \frac{1}{p-1}\sum_{a\in\Fpt} \Bigl|\frac{1}{\sqrt{p}}\sum_{x\in I}
  \psi_p(x+a\bar{x})\Bigr|^{4}
  =\frac{1}{p(p-1)}\sumsum_{\substack{x_1,\ldots,x_4\in
      I\\x_1^{-1}1+x_2^{-1}=x_3^{-1}+x_4^{-1}}} \psi_p(
  x_1+x_2-x_3-x_4) \\
  +O(|I|^4p^{-3}),
\end{multline*}
and a result of Bourgain and Garaev~\cite[Th. 1]{bourgain-garaev}
shows that the number of points of summation is $\ll |I|^{8/3+\eps}$
for any $\eps>0$, provided $|I|\leq p^{3/4}$, which is enough to
verify the hypothesis of Lemma~\ref{lm-tight}.
\end{remark}

\begin{proof}[Proof of Lemma~\ref{lm-tight}]
  We will verify Kolmogorov's criterion~(\ref{eq-kol}) for suitable
  $\alpha>0$ and $\delta>0$. We will deal with the case where
  $\mathcal{X}=(\xi_{p}(x))_{x\in\Fp}$ is parameterized by $\Fp$, the
  other being analogous.
\par
Let $L_p(t)=K^{\mathcal{X}}_p(t)$ for $p$ prime and $t\in [0,1]$.  Let
also
$$
\tilde{L}_p(t,\omega)=\frac{1}{\sqrt{p}}\sum_{0\leq x\leq
  pt}\xi_p(x,\omega)
$$
be the discontinuous analogue of $L_p(t)$. We first reduce the problem
to proving a moment estimate for $\tilde{L}_p(t)-\tilde{L}_p(s)$.
\par
We do this in two steps. First, if $|t-s|<1/p$, then the definition by
linear interpolation implies that
$$
|L_p(t,a)-L_p(s,a)|\leq \sqrt{p}|t-s| \leq |t-s|^{1/2}
$$
and hence
\begin{equation}\label{eq-0}
  \expect(|L_p(t)-L_p(s)|^{\alpha})\leq
  |t-s|^{\alpha/2}
\end{equation}
for any $\alpha>0$ and all primes $p$, which is fine as soon as
$\alpha>2$.
\par
Thus we assume from now on that $|t-s|\geq 1/p$. We then use the bound
$$
|L_p(t,a)-\tilde{L}_p(t,a)|\leq \frac{1}{\sqrt{p}},
$$
for all $t\in [0,1]$ and $a\in\Fpt$ (as in~(\ref{eq-desmooth})) and
deduce
\begin{align}
  \expect(|L_p(t)-L_p(s)|^{\alpha})&\ll
  \expect(|\tilde{L}_p(t)-\tilde{L}_p(s)|^{\alpha})+
  p^{-\alpha/2}\nonumber\\
  & \ll\expect(|\tilde{L}_p(t)-\tilde{L}_p(s)|^{\alpha})+
  |t-s|^{\alpha/2}
\label{eq-moms}
\end{align}
for any $\alpha\geq 1$, where the implied constant depends only on
$\alpha$. This shows that, provided $\alpha>2$, we
obtain~(\ref{eq-kol}) from the corresponding statement for
$\tilde{L}_p(t)$. We now begin to prove this.
\par
We denote
$$
X_p(t)=\sum_{|h|<p/2}\frac{\alpha'_p(h;t)}{\sqrt{p}}\st_h,
$$
the analogue of the random variables in~(\ref{eq-xp}) for intervals
$0\leq j\leq pt$ instead of $1\leq j\leq (p-1)t$.  For $s\leq t$ in
$[0,1]$, we will also denote by $I$ the interval $ps\leq x\leq pt$, of
length $|I|\asymp |t-s|p$ (recall that $|s-t|\geq 1/p$ now).
\par
We denote by $\delta_1>0$ and $\delta_2>0$ the parameters
in~(\ref{eq-short}). We first make the remark that we may replace
$\delta_1$ by any smaller positive number without affecting the
validity of~(\ref{eq-short}), so that we can assume that
$\delta_1<\delta_2$.
\par
We first claim that
\begin{equation}\label{eq-c1}
\expect(|\tilde{L}_p(t)-\tilde{L}_p(s)|^4)=
\expect(|X_p(t)-X_p(s)|^4)+ O(p^{-1/2}(\log p)^4).
\end{equation}
\par
This indeed follows from the assumption of Lemma~\ref{lm-tight}, by
the same method used in the proof of convergence of finite
distributions.
\par
Next, we claim that there exists $C\geq 0$ such that, for all $s$, $t$
in $[0,1]$ and all $p$, we have
\begin{equation}\label{eq-c2}
\expect(|X_p(t)-X_p(s)|^4)\leq C|t-s|^2
\end{equation}
(in particular, the sequence $(X_p(t))_{t\in [0,1]}$ is itself tight
by Kolmogorov's criterion). Indeed, we can use the fact that
$$
X_p(t)-X_p(s)=
\sum_{|h|<p/2}\frac{\alpha'_p(h;t)-\alpha'_p(h;s)}{\sqrt{p}} \st_h
$$
is $\sigma_p$-subgaussian, where
$$
\sigma_p^2=\frac{1}{p}\sum_{|h|<p/2}|\alpha'_p(h;t)-\alpha'_p(h;s)|^2
=\frac{1}{p}\sum_{ps\leq x\leq pt}1\ll |t-s|
$$
(by the discrete Plancherel formula). Since, for a
$\sigma$-subgaussian variable $N$, we have the
bound
$$
\expect(|N|^4)\leq 256\sigma^4
$$
(see~(\ref{eq-subg})), the claim follows.
\par
Combining~(\ref{eq-c1}) and~(\ref{eq-c2}), we get for any $\eta>0$ and
$\eps>0$ the bound
$$
  \expect(|\tilde{L}_p(t)-\tilde{L}_p(s)|^4)
  \ll |t-s|^{1+\eta-\eps}
$$
for all $s$ and $t$ such that $|t-s|\geq p^{-1/(2(1+\eta))}$, where
the implied constant depends only on $\eps$. For suitable $\eta>0$ and
$\eps>0$, this gives
\begin{equation}\label{eq-bound-1}
  \expect(|\tilde{L}_p(t)-\tilde{L}_p(s)|^4)
  \ll |t-s|^{1+\delta'}
\end{equation}
for all $s$, $t$ such that $|t-s|\geq p^{-1/2+\delta_1}$, where
$\delta'>0$.
\par
Next, suppose that $p^{-1}\leq |t-s|\leq p^{-1/2-\delta_1}$. We then
note that the trivial bound
$$
|\tilde{L}_p(t,\omega)-\tilde{L}_p(s,\omega)|\leq
\frac{|I|}{\sqrt{p}}\ll p^{1/2}|t-s|\ll p^{-\delta_1}
$$
leads to
$$
\expect(|\tilde{L}_p(t)-\tilde{L}_p(s)|^{\alpha}) \ll
p^{-\alpha\delta_1}\ll |t-s|^{2}
$$
provided $\alpha>2\delta_1^{-1}$. Thus we get
\begin{equation}\label{eq-bound-2}
  \expect(|\tilde{L}_p(t)-\tilde{L}_p(s)|^{\alpha_1})\ll |t-s|^{2},
\end{equation}
for $\alpha_1=4\delta_1^{-1}$ and $p^{-1}\leq |t-s|\leq
p^{-1/2-\delta_1}$.
\par
Finally, assume that
$$
p^{-1/2-\delta_1}\leq |s-t|\leq p^{-1/2+\delta_1},
$$
so that
$$
p^{1/2-\delta_1}\ll |I|\ll p^{1/2+\delta_1}.
$$
\par
By~(\ref{eq-short}), with $\alpha>0$ as in that bound, we deduce that
\begin{equation}\label{eq-bound-3}
  \expect(|\tilde{L}_p(t)-\tilde{L}_p(s)|^{\alpha})\ll 
  p^{-1/2-\delta_2}\ll |t-s|^{(1/2+\delta_2)/(1/2+\delta_1)}\leq |t-s|^{1+\delta''}
\end{equation}
for some $\delta''>0$, since we assumed that $\delta_2>\delta_1$.
\par
We can now combine~(\ref{eq-bound-1}),~(\ref{eq-bound-2})
and~(\ref{eq-bound-3}). 
All ranges of $|t-s|$ are covered by the combination of the three
bounds, and we have a suitable inequality for each range. The
exponents on both sides of the inequalities do not necessarily match,
however. But since we have
$$
\tilde{L}_p(t,\omega)\ll (\log p)
$$
uniformly (by the completion method, as in~(\ref{eq-complete}), which
applies thanks to the assumption (1) of Lemma~\ref{lm-tight} that the
Fourier transforms are uniformly bounded by $H$), we can replace the
exponent of $|\tilde{L}_p(t)-\tilde{L}_p(s)|$ by
$\max(4,\alpha_1,\alpha)>0$, which uniformizes the exponent on the
left, at the cost of a power of $\log p$. Since $|t-s|\leq 1$, we can
also replace the exponent of $|t-s|$ by
$$
0<\delta<\min(1,\delta',\delta'')
$$
and obtain then
$$
\expect(|\tilde{L}_p(t)-\tilde{L}_p(s)|^{\alpha})\ll (\log
p)^{C}|t-s|^{1+\delta}
$$
for some $\alpha>0$, $\delta>0$, $C\geq 0$ and for all $s$, $t$ with
$1/p\leq |t-s|\leq 1$. In this range, this means that
$$
\expect(|\tilde{L}_p(t)-\tilde{L}_p(s)|^{\alpha})\ll |t-s|^{1+\delta-\eps}
$$
for any $\eps>0$, where the implied constant depends only on
$\eps>0$. Together with the introductory reduction, this verifies
Kolmogorov's criterion.
\end{proof}

\section{Applications}
\label{sec-applis}

The fact that, for any fixed $t_0\in [0,1]$, there is a limiting
distribution for $K_p(t_0)$ (or for the Birch process at $t_0$) is
already interesting, although it is only the simplest case of
convergence of finite distributions. We can then use known results on
sums of independent variables to deduce some interesting properties of
the corresponding partial sums.  We study here only the tail behavior
of the limiting distribution at $t_0$, and get:

\begin{proposition}\label{pr-tail}
  Let $t_0\in ]0,1[$ be given. There exists a constant $c(t_0)>0$ such
  that for any $A>0$, we have
\begin{multline*}
  \frac{1}{c(t_0)}\exp(-\exp(c(t_0)A)) \\\leq \lim_{p\ra +\infty}
  \frac{1}{p-1}|\{a\in\Fpt\,\mid\, |\Reel(K_p(t_0))|\geq A\}| \\\leq
  c(t_0)\exp\Bigl(-\exp\Bigl(\frac{A}{c(t_0)}\Bigr)\Bigr),
\end{multline*}
if $t_0\not=1/2$, and
\begin{multline*}
  \frac{1}{c(t_0)}\exp(-\exp(c(t_0)A)) 
 \\\leq \lim_{p\ra +\infty}
  \frac{1}{p-1}|\{a\in\Fpt\,\mid\, |\Imag(K_p(t_0))|\geq A\}| \\\leq
  c(t_0)\exp\Bigl(-\exp\Bigl(\frac{A}{c(t_0)}\Bigr)\Bigr).
\end{multline*}
\end{proposition}

\begin{remark}
  (1) For $t_0=1/2$, we have $\Reel(K_p(1/2))=\demi \st_0$, so that
  the upper-bound estimate for $|\Reel(K_p(1/2))|$ holds trivially,
  but the lower-bound fails for $A\geq 1$.
\par
(2) In fact, since we know the exact distribution of $\kpath(t_0)$,
one could probably improve this result with more work, using moment
methods similar to those used by Granville and
Soundararajan~\cite{granville-sound-zeta}.
\end{remark}

\begin{proof}
  We consider the real part, the imaginary part being handled
  similarly.  By convergence of finite distributions, it is equivalent
  to prove that there exists $c(t_0)>0$ such that
$$
  c(t_0)^{-1}\exp(-\exp(c(t_0)A))\leq 
\proba(|\Reel(\kpath(t_0))|\geq A)
\leq   c(t_0)\exp(-\exp(c(t_0)^{-1}A))
$$
for $t_0\in ]0,1[$, $t_0\not=1/2$.
\par
We begin with the upper-bound. We will use the martingale method
explained by Ledoux and Talagrand~\cite[\S 1.3]{ledoux-talagrand}, but
there are other options (e.g., the work of Montgomery and
Odlyzko~\cite[Th. 2]{m-o} or probabilistic methods similar to those in
the later Proposition~\ref{pr-tail2}, along the lines of
Montgomery-Smith's work for Rademacher series~\cite{ms}).
\par
We write
$$
\Reel(\kpath(t_0))=\sum_{n\geq 0}X_n,
$$
where 
\begin{gather*}
  X_0=t_0\st_0 ,\quad\text{ and } \quad X_n=\frac{\sin 2\pi nt_0}{2\pi n}
(\st_n+\st_{-n})
\end{gather*}
for $n\geq 1$. The random variables $(X_n)_{n\geq 0}$ are independent,
bounded, integrable with expectation zero. In particular, for any
$N\geq 1$, the sum
$$
\sum_{0\leq n\leq N}X_n
$$
is an example of a \emph{sum of martingale differences} as described
in~\cite[p. 31]{ledoux-talagrand} (with $d_i=X_i$), with expectation
$0$. By~\cite[Lemma 1.8]{ledoux-talagrand}, we have
$$
\proba\Bigl(\Bigl|\sum_{0\leq n\leq N}X_n\Bigr|>A\Bigr) \leq
16\exp\Bigl(-\exp\Bigl(\frac{A}{4a_N}\Bigr)\Bigr)
$$
for any $A>0$, where
$$
a_N=\max_{0\leq i\leq N} (i+1)\|X_i\|_{L^{\infty}(\Omega)}.
$$
\par
For all $N\geq 1$, we have
\begin{align*}
  a_N\leq a&=\sup_{n\geq 0} (n+1)\|X_n\|_{L^{\infty}(\Omega)}\\
  &= \max\Bigl(2|t_0|,\sup_{n\geq 1} \frac{2n}{\pi n}|\sin(2\pi
  nt_0)|\Bigr)\\
  &=\max\Bigl(2|t_0|,\frac{2}{\pi}\sup_{n\geq 1} |\sin(2\pi
  nt_0)|\Bigr)\leq\max\Bigl(2|t_0|,\frac{2}{\pi}\Bigr)
\end{align*}
and in fact, if $t_0$ is irrational, then $a=\max(2|t_0|,2/\pi)$ (and
otherwise it can be computed quite easily as a function of the
denominator of $t_0$).
\par
Thus, for any $N\geq 1$ and $A>0$, we get
$$
\proba\Bigl(\Bigl|\sum_{0\leq n\leq N}X_n\Bigr|>A\Bigr) \leq
16\exp\Bigl(-\exp\Bigl(\frac{A}{4a}\Bigr)\Bigr).
$$
\par
We can now easily let $N\ra +\infty$: fix $A>0$ and let $\eps>0$;
there exists $N\geq 1$ such that
$$
\proba\Bigl(\Bigl|\sum_{n>N}X_n\Bigr|>\eps\Bigr) \leq \eps
$$
(convergence in probability of the partial sums of $\kpath(t_0)$,
which follows from Proposition~\ref{pr-limit}) and thus
$$
\proba\Bigl(\Bigl|\sum_{n\geq 0}X_n\Bigr|>A\Bigr) \leq \eps+
\proba\Bigl(\Bigl|\sum_{0\leq n\leq N}X_n\Bigr|>A-\eps\Bigr) \leq
\eps+16\exp\Bigl(-\exp\Bigl(\frac{A-\eps}{4a}\Bigr)\Bigr).
$$
\par
Letting $\eps\ra 0$ gives the desired upper-bound for the real part,
in a rather precise and explicit form.
The lower-bound can here be derived elementarily. Write
$$
X_n=r_nY_n
$$
with
\begin{align*}
r_0=t_0,\quad\quad r_n=\frac{\sin 2\pi nt_0}{2\pi n},\text{ for }
n\geq 1\\
Y_0=\st_0,\quad\quad Y_n=\st_n+\st_{-n}\text{ for } n\geq 1.
\end{align*}
\par
Using symmetry and independence of the variables $(Y_n)$, we have
$$
\proba\Bigl(\Reel(\kpath(t_0))> A\Bigr)\geq \frac{1}{4}\prod_{1\leq
  n\leq N} \proba(|Y_n|\geq 1)=\frac{1}{4}u^{-N-1}
$$
where $u=\proba(Y_1\geq 1)=\proba(Y_1\leq -1)>0$, for any $N$ such
that
$$
\sum_{1\leq n\leq N}|r_n|> A.
$$
\par
One can find such an $N$ with $N\ll \exp(\delta_3 A)$ for some
$\delta_3>0$, where $\delta_3$ and the implied constant depend on
$t_0$ (e.g., again equidistribution of $nt_0$ in $\Rr/\Zz$, if $t_0$
is irrational, or by periodicity, using the fact that $t_0$ is assumed
to be $\not=1/2$). This gives the desired lower-bound.
\par
As already mentioned, the imaginary part is handled similarly; note
that there is no exception similar to $t_0=1/2$ because the sequence
$(\cos(2\pi nt_0)-1)/(2\pi n)$ vanishes identically only if $t_0\in
\{0,1\}$.
\end{proof}

We now consider Theorem~\ref{th-tails}, which is an example of
application requiring convergence in law in $C([0,1])$. Since the norm
map $\varphi\mapsto \|\varphi\|_{\infty}$ is continuous on $C([0,1])$,
it follows formally from the definition of convergence in law that the
random variables $(\alpha,a)\mapsto \|\ekp_t(t;\alpha,a)\|_{\infty}$
(resp. $a\mapsto \|K^{\mathcal{X}}_p(a)\|_{\infty}$ for Birch sums)
converge in law to the random variable $\|\kpath\|_{\infty}$ as $p\ra
+\infty$. (Recall that the $L^{\infty}$-norm refers to the space
$C([0,1])$, and not to the space $L^{\infty}(\Omega)$).
\par
Moreover, since the maximum of the modulus along a segment
in $\Cc$ is achieved at one of the end points, we have
$$
\|\ekp_p(\cdot;\alpha,a)\|_{\infty}= \max_{1\leq j\leq p-1}
\frac{1}{\sqrt{p}} \Bigl|\sum_{1\leq x\leq
  j}\psi_p(\alpha(ax+\bar{x}))\Bigr|,
$$
resp.
$$
\|K^{\mathcal{X}}_p(\cdot;a)\|_{\infty}= \max_{0\leq j\leq p-1}
\frac{1}{\sqrt{p}} \Bigl|\sum_{0\leq x\leq j}\psi_p(ax+x^3)\Bigr|.
$$
\par
Hence, defining $\mu$ as the distribution of $\|\kpath\|_{\infty}$,
Theorem~\ref{th-main2} gives:

\begin{proposition}\label{pr-mu}
  There exists a probability measure $\mu$ on $[0,+\infty[$ such that
  for any bounded continuous function $f$ on $[0,+\infty[$, we have
$$
\lim_{p\ra +\infty} \frac{1}{p-1} \sum_{a\in\Fpt} f\Bigl(\max_{0\leq
  j\leq p-1}\frac{1}{\sqrt{p}} \Bigl|\sum_{0\leq x\leq
  j}\psi_p(ax+x^3)\Bigr| \Bigr)=\int f(x)d\mu(x),
$$
and
$$
\lim_{p\ra +\infty} \frac{1}{(p-1)^2}
\sumsum_{(\alpha,a)\in\Fpt\times\Fpt} f\Bigl(\max_{0\leq j\leq
  p-1}\frac{1}{\sqrt{p}} \Bigl|\sum_{0\leq x\leq
  j}\psi_p(\alpha(ax+\bar{x}))\Bigr| \Bigr)=\int f(x)d\mu(x).
$$
\end{proposition}

We expect of course that the same holds for Kloosterman sums without
the average over $\alpha$.
\par 
Using this result, we can now prove Theorem~\ref{th-tails}, by getting
suitable tail bounds for the limiting distribution of $\kpath$. More
precisely, Theorem~\ref{th-tails} follows from the following:

\begin{proposition}\label{pr-tail2}
  There exists $c>0$ such that
$$
c^{-1}\exp(-\exp(cA)) \leq \proba\Bigl(\|\kpath\|_{\infty}\geq A\Bigr)
\leq c\exp(-\exp(c^{-1}A))
$$
for any $A>0$.
\end{proposition}

\begin{proof}
  The lower bound is an immediate consequence of the lower-bound in
  Proposition~\ref{pr-tail} (for $t_0=1/2$ and the imaginary part,
  say).
\par
For the upper bound, we will apply some results of probability theory
in the Banach space $C_{\Rr}([0,1])$ of real-valued continuous
functions on $[0,1]$, dealing separately with the real and imaginary
parts of $\kpath$.
\par
We view $\Reel(\kpath)$ as the sum
$$
\Reel(\kpath)=X_0\varphi_0+\sum_{n\geq 1} X_n\varphi_n
$$
where $(X_n)_{n\geq 0}$ are independent, symmetric, real-valued random
variables with
$$
|X_0|\leq 2,\quad\quad |X_n|\leq 4\text{ for } n\geq 1,
$$
and
$$
\varphi_0(x)=x,\quad\quad \varphi_n(x)=\frac{\sin(2\pi nx)}{2\pi
  n}\text{ for } n\geq 1
$$
are vectors in $C_{\Rr}([0,1])$.
\par
By a result of Talagrand~\cite[Remarks after Th. 13.2,
(13.12)]{talagrand} (or almost equivalently an adaptation of the main
theorem of~\cite{d-ms} to the variables $(X_n)_{n\geq 1}$ instead of
Rademacher variables, replacing the crucial theorem of
Talagrand~\cite[Theorem A]{d-ms} used in its proof
by~\cite[Th. 13.2]{talagrand}) we have
\begin{align*}
  \proba\Bigl(\|\Reel(\kpath)\|_{\infty}> m + 4N(s)\Bigr)&=
  \proba\Bigl(\Bigl\|\sum_{n\geq 0}X_n\varphi_n\Bigr\|_{\infty}> m +
  4N(s)\Bigr)\\&\leq 4\exp(-s^2/16)
\end{align*}
for any $s>0$, where
we have denoted by $m$ a median of the random variable
$\|\Reel(\kpath)\|_{\infty})$ and the function $N(s)$ is the function
denoted $K_{1,2}^w((x_n),t)$ in~\cite{d-ms} (or $\kappa(t)$
in~\cite[p. 199]{talagrand}) for the sequence $x_n=\varphi_n$ (the
factor $4$ in front of $N(s)$ is due to the assumption
in~\cite{talagrand} that the random variables are bounded by
$1$). 
\par
We have
$$
\proba(\|\Reel(\kpath)\|_{\infty}>A)\leq
\proba\Bigl(\|\Reel(\kpath)\|_{\infty}> m +4N(s)\Bigr)
$$
for any $s>0$ such that $m+4N(s)\leq A$.
\par
Now we claim that there exists $c>0$ such that
$$
N(s)\leq c\log(cs)
$$
for all $s\geq 1$.
\par
This is proved in Lemma~\ref{lm-interp} below. Assuming this property,
we take
$$
s=\frac{1}{c}\exp\Bigl(\frac{A-m}{4c}\Bigr).
$$ 
\par
We may clearly assume that $A$ is large enough so that $s\geq 1$ (the
desired estimate being trivial otherwise). Then $s\geq 1$ also
satisfies $m+N(s)\leq A$, and we deduce that
$$
\proba(\|\Reel(\kpath)\|_{\infty}>A)\leq 4\exp(-s^2/16),
$$
which has the required form for the real part of $\kpath$. A similar
argument applies to the imaginary part, finishing the proof.
\end{proof}

We now prove the estimate we used:

\begin{lemma}\label{lm-interp}
  With notation as in the proof above, there exists $c>0$ such that we
  have
$$
N(s)\leq c\log cs
$$
for all $s\geq 1$.
\end{lemma}

\begin{proof}
  By definition (see~\cite[p. 2046]{d-ms}), we have
$$
N(s)=\sup_{\|\lambda\|\leq 1}N_{\lambda}(s)
$$
where $\lambda$ runs over elements of the dual space of the real
Banach space $C_{\Rr}([0,1])$ with Banach norm $\leq 1$, and
$$
N_{\lambda}(s)=\inf\{ \|x_1\|_1+ s\|x_2\|_2\,\mid\,
(\lambda(\varphi_n))=
x_1+x_2\},
$$
where $x_i\in\ell_i$, and $\|x_i\|_i$ is the $\ell_i$-norm in the
Banach space $\ell_i$ of $i$-th power summable sequences of real
numbers.
\par
As suggested by a result of Holmstedt (see~\cite[Th. 4.1]{holmstedt}
and~\cite[p. 518]{ms}) which gives a two-sided equivalent to
$N_{\lambda}(s)$, we define
$$
x_{1,n}=\begin{cases}
  \lambda(\varphi_n)&\text{ if } 0\leq n\leq s^2\\
  0&\text{ if } n>s^2
\end{cases}
$$
(which therefore determines $x_2$). We then have
$$
N_{\lambda}(s)\leq \|x_1\|_1+s\|x_2\|_2.
$$
\par
Since the linear form $\lambda$ has norm $\leq 1$, we have
$$
|\lambda(\varphi_0)|\leq \|\varphi_0\|_{\infty}= 1
$$
and
$$
|\lambda(\varphi_n)|=\Bigl|\frac{1}{2\pi n} \lambda(s_n)\Bigr|\leq
\frac{1}{2\pi n}
$$
(where $s_n(x)=\sin(2\pi nx)$) for $n\geq 1$. Thus we get
$$
N_{\lambda}(s)\leq 1+\sum_{1\leq n\leq s^2}\frac{1}{2\pi n}+
s\Bigl(\sum_{n>s^2}\frac{1}{4\pi^2n^2}\Bigr)^{1/2} \leq c\log cs
$$
for some $c>0$ and all $s\geq 1$. 
\end{proof}



\section{Final remarks}
\label{sec-other}

\subsection{Variants}

It is clear that the general setting admits many variations. When
looking at other families of one-variable exponential sums, a number
of complications may arise. For instance, in many situations, the
analogue of Lemma~\ref{lm-katz} will have to take into account the
inter-dependencies of the monodromy groups of the analogues of the
shifted Kloosterman sheaves, and of course the estimates of short sums
necessary for tightness are not always known.
\par
One can also consider families of exponential sums parameterized by
multiplicative characters. In that case, we need to exploit Katz's
recent definition of an analogue of the monodromy group for Mellin
transforms over finite fields~\cite{katz-mellin}, in order to have
statements similar to Lemma~\ref{lm-katz}.  But tightness is then
sometimes easier to prove, because of the small ``multiplicative
energy'' of intervals.
\par
Yet another variation would involve re-parameterizing the order of
summation in Kloosterman (or Birch) sums. The most natural way to do
this is to pick a primitive root $g\in\Fpt$, and to consider the
continuous path interpolating between the partial sums
$$
\frac{1}{\sqrt{p}}\sum_{0\leq m\leq (p-2)t}\psi_p(ag^m+g^{-m}),
$$
which has the same start and end points as $K_p(t)$. 
\par
In all of these situations, one can hope to have similar results as
those in this paper. We will come back to such situations in a later
work.
\par
On the other hand, answering the following other very natural question
seems well out of reach of current methods:

\begin{problem}
Consider the random variables
$$
\tilde{K}_X(t)\,:\, p\mapsto K_p(t,1)
$$
on the finite probability space
$$
(\{p\text{ prime }\leq X\},\text{uniform measure}).
$$
\par
Does $(\tilde{K}_X(t))_{t\in [0,1]}$ also converge in law to
$(\kpath(t))_{t\in [0,1]}$?
\end{problem}

This is the analogue for Kloosterman paths of the famous horizontal
Sato-Tate conjecture for Kloosterman sums, which remains completely
open.

Convergence in finite distributions for this sequence of random variables would follow from the general horizontal equidistribution conjecture of \cite{sawin}, by the same argument as the proof of Theorem \ref{th-main}.

\subsection{Similar works}

We conclude with a brief mention of some related works.

(1) From the probabilistic point of view, it is worth observing that
convergence in law of processes (in $C([0,1])$) related to
number-theoretic quantities has already been discovered in a few
cases. The best-known is probably Billingsley's generalization of the
Erd\"os-Kac Theorem, which gives convergence to Brownian motion of
suitable normalized counts of primes dividing integers in varying
intervals (see~\cite[Ch. 4, \S 17]{billingsley}). One can also mention
Bagchi's probabilistic interpretation of Voronin's Universality
Theorem (where convergence in law happens in a space of holomorphic
functions, see~\cite[\S 0.2]{bagchi}).

(2) The papers on paths of exponential sums of Lehmer and
Loxton~\cite{lehmer, loxton1} consider situations rather different,
where quite precise asymptotic evaluation is possible, e.g., in terms
of Fresnel integrals for incomplete quadratic Gauss sums.
In~\cite{loxton2}, Loxton considers (roughly) sums of $\psi_p(g(x))$
where $g$ ranges uniformly over polynomials of some degree $d$ over
$\Fp$, and obtains some limit theorems which are however of a rather
different kind as ours.

(3) Another important case is that of the classical character sums
$$
S(\chi;N)=\sum_{1\leq n\leq N}\chi(n)
$$
for $\chi$ a non-trivial Dirichlet character modulo $q\geq 1$ and
$1\leq N\leq q$. These have been studied extensively from different
perspectives, going back at least as far as Littlewood and Paley.  The
papers of Granville and Soundararajan~\cite{granville-sound}, and of
Bober and Goldmakher~\cite{bober-goldmakher} consider (among other
things!)  the distribution of
$$
M(\chi)=\max_{1\leq N\leq q}S(\chi;N),
$$
the largest modulus of these sums as $N$ varies (compare for
instance~\cite[Th. 1.3]{bober-goldmakher} and Theorem~\ref{th-tails}).
More recently, Bober, Goldmakher, Granville and
Koukoulopoulos~\cite{bggk} have proved the analogue of the main
results of this paper for these sums. The limiting process is very
different however, as it has to take into account the multiplicative
structure of the Fourier coefficients.




\begin{thebibliography}{CC}


\bibitem{bagchi} B. Bagchi: \textit{Statistical behaviour and
    universality properties of the Riemann zeta function and other
    allied Dirichlet series}, PhD thesis, Indian Statistical
  Institute, Kolkata, 1981; available at
  \url{library.isical.ac.in/jspui/handle/10263/4256}


\bibitem{billingsley} P. Billingsley: \textit{Convergence of
    probability measures}, 2nd Edition, Wiley 1999.

\bibitem{birch}
B.J. Birch: \textit{How the number of points of an elliptic curve over
a fixed prime field varies}, J. London Math. Soc. 43 (1968), 57--60. 




\bibitem{bober-goldmakher} J.W. Bober and L. Goldmakher: \textit{The
    distribution of the maximum of character sums}, Mathematika 59
  (2013), 427--442.

\bibitem{bggk} J.W. Bober, L. Goldmakher, A. Granville and
  D. Koukoulopoulos: \textit{The frequency and the structure of large
    character sums}, preprint (2014).



\bibitem{bourgain-garaev} J. Bourgain and M.Z. Garaev: \textit{Sumsets
  of reciprocals in prime fields and multilinear Kloosterman sums},
\url{arXiv:1211.4184}. 







\bibitem{weilii}
P. Deligne: \textit{La conjecture de Weil, II}, Publ. Math. IH\'ES 52
(1980), 137--252.



\bibitem{d-ms}  S. J. Dilworth and S. J. Montgomery-Smith: \textit{The
    distribution of vector-valued Rademacher series},  Annals of
  Prob. 21 (1993), 2046--2052.







\bibitem{sums-of-products} \'E. Fouvry, E. Kowalski, Ph. Michel:
  \textit{A study in sums of products}, preprint (2014).

\bibitem{fouvry-michel} \'E. Fouvry and Ph. Michel: \textit{Sur
    certaines sommes d'exponentielles sur les nombres premiers},
Ann. Sci. \'Ecole Norm. Sup. (4) 31 (1998), 93--130.






   

\bibitem{granville-sound-zeta} A. Granville and K. Soundararajan:
  \textit{Extreme values of $\zeta(1+it)$}, in ``The Riemann zeta
  function and related themes: papers in honor of Professor
  K. Ramachandra'', Ramanujan Math. Soc.  Lect. Notes Ser. 2 (2006),
  65--80.

\bibitem{granville-sound} A. Granville and K. Soundararajan:
  \textit{Large character sums: pretentious characters and the
    P\'olya-Vinogradov theorem}, Journal of the AMS 20 (2007), 357--
  384.



\bibitem{holmstedt} T. Holmstedt: \textit{Interpolation of
    quasi-normed spaces}, Math. Scand. 26 (1970), 177--199. 

\bibitem{irving} A.J. Irving: \textit{The divisor function in
    arithmetic progressions to smooth moduli}, preprint
  \url{arXiv:1403.8031v2} 


 


\bibitem{ant} H. Iwaniec and E. Kowalski: \textit{Analytic number
    theory}, A.M.S. Coll. Publ. 53 (2004).

\bibitem{kahane} J-P. Kahane: \textit{Some random series of
    functions}, Cambridge Studies Pure Math. 5, C.U.P (1985).




\bibitem{katz-mg} N.M. Katz: \textit{On the monodromy attached to
    certain families of exponential sums}, Duke Math. J. 54 (1987),
  41--56.

\bibitem{katz-gkm} N.M. Katz: \textit{Gauss sums, Kloosterman sums and
    monodromy groups}, Annals of Math. Studies 116, Princeton
  Univ. Press (1988).

\bibitem{katz-esde} N.M. Katz: \textit{Exponential sums and
    differential equations}, Annals of Math. Studies 124, Princeton
  Univ. Press (1990).




\bibitem{katz-mellin} N.M. Katz: \textit{Convolution and
    equidistribution: Sato-Tate theorems for finite field Mellin
    transforms}, Annals of Math. Studies 180, Princeton Univ. Press
  (2012).




\bibitem{kloostermania} E. Kowalski: \textit{The Kloostermania page},
  \url{blogs.ethz.ch/kowalski/the-kloostermania-page/} 






\bibitem{ledoux-talagrand} M. Ledoux and M. Talagrand:
  \textit{Probability in Banach spaces: isoperimetry and processes},
  Ergebnisse der Math. 23 (1991).

\bibitem{lehmer} D.H. Lehmer: \textit{Incomplete Gauss sums},
  Mathematia 23 (1976), 125--135.

\bibitem{livne} R. Livn\'e: \textit{The average distribution of cubic
    exponential sums}, J. reine angew. Math. 375-376 (1987), 362--379.

\bibitem{loxton1} J.H. Loxton: \textit{The graphs of exponential
    sums}, Mathematika 30 (1983), 153--163.

\bibitem{loxton2} J.H. Loxton: \textit{The distribution of exponential
    sums}, Mathematika 32 (1985), 16--25.



\bibitem{m-o} H. Montgomery and A. Odlyzko: \textit{Large deviations
    of sums of independent random variables}, Acta Arith. 49 (1988),
  427--434.

\bibitem{ms} S.J. Montgomery-Smith: \textit{The distribution of
    Rademacher sums}, Proc. Amer. Math. Soc. 109 (1990), 517--522.

 




\bibitem{revuz-yor} D. Revuz and M. Yor: \textit{Continuous
    Martingales and Brownian Motion}, 3rd ed., Springer-Verlag,
  Berlin, 1999.



\bibitem{sawin} W.F. Sawin: \textit{A Tannakian Category of Arithmetic
    Exponential Sums}, in progress (expected 2014).





\bibitem{talagrand} M. Talagrand: \textit{Concentration of measure and
    isoperimetric inequalities in product spaces},
  Publ. Math. I.H.É.S. 81 (1995), 73--205.



\end{thebibliography}
\end{document}